\documentclass[11pt,oneside,leqno]{amsart}
\allowdisplaybreaks
\usepackage[margin=2.5cm]{geometry}
\usepackage{multirow,tabularx}
\newcolumntype{Y}{>{\centering\arraybackslash}X}

\usepackage{pbox}
\usepackage{parskip}
\usepackage{eucal}
\usepackage{amscd}
\usepackage{amssymb}
\usepackage{amsmath}
\usepackage{amsfonts}
\usepackage{amsopn}
\usepackage{latexsym}
\usepackage{tikz}
\usepackage{comment}
\usepackage{ulem}
\usepackage{hyperref}
\usepackage[nodisplayskipstretch]{setspace}
\usepackage{caption}
\usepackage[OT2,T1]{fontenc}
\DeclareSymbolFont{cyrletters}{OT2}{wncyr}{m}{n}
\DeclareMathSymbol{\Sha}{\mathalpha}{cyrletters}{"58}
\newcommand{\genlegendre}[4]{%
  \genfrac{(}{)}{}{#1}{#3}{#4}%
  \if\relax\detokenize{#2}\relax\else_{\!#2}\fi
}

\usepackage{enumerate}
\usepackage[utf8]{inputenc}

\makeatletter
\@namedef{subjclassname@2010}{%
  \textup{2010} Mathematics Subject Classification}
\makeatother

\newtheorem{thm}{Theorem}[section]
\newtheorem{cor}[thm]{Corollary}
\newtheorem{lem}[thm]{Lemma}

\makeatletter
\newcommand{\BIG}{\bBigg@{2}}
\newcommand{\vast}{\bBigg@{3}}
\newcommand{\Vast}{\bBigg@{5}}
\makeatother

\numberwithin{equation}{section}

\begin{document}
\setlength{\arrayrulewidth}{0.1mm}



\title[Class Number Divisibility and Selmer Rank]{Class number divisibility of $\mathbb{Q}(\sqrt{3p}, \sqrt{m-21pn^{2}})$ constructed from elliptic curves of $2$-Selmer rank exactly $1$}

\author[Debopam Chakraborty]{Debopam Chakraborty}
\address{Department of Mathematics\\ BITS-Pilani, Hyderabad campus\\
Hyderabad, INDIA}
\email{debopam@hyderabad.bits-pilani.ac.in}

\author[Vinodkumar Ghale]{Vinodkumar Ghale}
\address{Department of Mathematics\\ BITS-Pilani, Hyderabad campus\\
Hyderabad, INDIA}
\email{p20180465@hyderabad.bits-pilani.ac.in}

\author[Md Imdadul Islam]{Md Imdadul Islam}
\address{Department of Mathematics\\ BITS-Pilani, Hyderabad campus\\
Hyderabad, INDIA}
\email{p20200059@hyderabad.bits-pilani.ac.in}
\date{}

\subjclass[2020]{Primary 11R29, 11G07; Secondary 11G05, 11R11}
\keywords{Class number; Elliptic curve; Selmer group.}

\maketitle

\section*{Abstract}
\noindent The class number divisibility problem for number fields is one of the classical problems in algebraic number theory, which originated from Gauss' class number conjectures. The relation between the points on an elliptic curve and class number divisibility of a number field has been explored through the works of various mathematicians. Here, we explicitly construct an unramified abelian extension of a bi-quadratic field generated from points of a certain type of elliptic curve. Moreover, showing the $2$-Selmer rank of the said elliptic curve as $1$, we also construct an infinite family of bi-quadratic fields of even class number.  

\section{Introduction}
\noindent The class group of a number field $K$ measures how much the ring of integers $\mathcal{O}_{K}$ deviates from having unique factorization into irreducible elements. The ideal class group of a number field is the quotient of the group of fractional ideals of $K$ by the group of principal fractional ideals. The ideal class group of a number field $K$ can be identified with the Galois group of the maximal unramified abelian extension of $K$. Hence, an unramified abelian extension plays an important role in the class number divisibility problem for a number field.\\
\noindent Gauss was the first to formalize the definition of class groups for quadratic number fields through the language of binary quadratic forms. He conjectured that for $d < 0$, the class number $h(d)$ of $\mathbb{Q}(\sqrt{d})$ tends to infinity as $-d$ tends to infinity. Heilbronn proved this result in \cite{Heil}. The general Gauss' class number problem for imaginary quadratic fields was solved through the works of Goldfield \cite{Gold1} \cite{Gold2}, Gross, and Zagier \cite{Gross}. Several mathematicians studied ideal class groups of a number field through different tools that originate from the arithmetic of the elliptic curve. In \cite{Honda1} and \cite{Honda2}, T. Honda used elliptic curves to construct infinitely many quadratic fields with class numbers divisible by $3$. A.Sato followed Honda's work with a geometric interpretation and extended the results to the cases for primes $5$ and $7$ in \cite{Sato1}. In \cite{Soleng1}, R. Soleng gave a class number divisibility condition for quadratic fields with ideal class groups isomorphic to the torsion group of certain elliptic curves. Lemmermeyer \cite{Lem1} constructed an unramified quadratic extension of cubic fields generated from points of elliptic curves. A similar approach for bi-quadratic fields was taken by Chakraborty et al. in \cite{Chak1}. In a fundamental work of Coates and Sujatha \cite{Coat}, it was shown that the fine Selmer group, a subgroup of the classical Selmer group, satisfies stronger local conditions and shows stronger finiteness properties than the regular Selmer group. In \cite{Kumar}, Kumar Murty et al. proved the similarity between the growth of the $p$-rank of the Fine Selmer group and the growth of the ideal class group of a number field.\\
\noindent In this work, we construct a bi-quadratic field from the points of an elliptic curve and then simultaneously try to understand the $2$-part of the Selmer and Shafarevich Tate group of the said curve and also the class number divisibility problem of the bi-quadratic field constructed from the points of that curve. The main result of this work is the following theorem, which explicitly constructs an unramified abelian extension of a bi-quadratic field generated from an elliptic curve of $2$-Selmer rank exactly 1. We note that the elliptic curve here is distinctively different from the curve $y^{2} = x^{3} + m$ that was used in both \cite{Lem1} and \cite{Chak1}.
\begin{thm}\label{mainthm}
    Let $E_{p}: y^{2} = (x+6p) (x-9p) (x-18p)$ be an elliptic curve where $p$ is a prime such that  either $p \equiv 17, 113 \pmod {120}$. Then the $2$-Selmer rank of $E_{p}$ is exactly $1$. Moreover, if $\big(\frac{r}{t^{2}}, \frac{s}{t^{3}}\big)$ is an arbitrary point on $E_{p}$, then there exists an unramified quadratic extension of the bi-quadratic field $K = \mathbb{Q}(\sqrt{3p}, \sqrt{r-21pt^{2}})$ whenever $t$ is even and $\gcd(s,3p) = 1$. 
\end{thm}

\noindent The following result deals with the case when $p$ is a prime of the form $p \equiv 53, 77 \pmod {120}$. The result ensures the rank of $E_{p}$ to be zero. 
\begin{cor}\label{cor1}
        Let $E_{p}: y^{2} = (x+6p) (x-9p) (x-18p)$ be an elliptic curve where $p$ is a prime such that $p \equiv 53, 77 \pmod {120}$. Then the $2$-Selmer rank of $E_{p}$ is $0$.
\end{cor}
\noindent The following result constructs an infinite family of bi-quadratic fields of even class number generated from $E_{p}$ for a fixed prime $p \equiv 17, 113 \pmod {120}$. 
\begin{cor}\label{cor2}
    If $P_{0} = \Big(\frac{r_{0}}{t_{0}^{2}}, \frac{s_{0}}{t_{0}^{3}}\Big)$ is an arbitrary point on $E_{p}$ such that $t_{0}$ is even and $\gcd(s_{0}, 3p) = 1$, then there exists an infinite family of the bi-quadratic fields $K_{i} = \mathbb{Q}(\sqrt{3p}, \sqrt{r_{i}-21pt_{i}^{2}})$ with even class number, $i = 0,1,2,...$ and $P_{i} = 2^{i} P_{0} =  \Big(\frac{r_{i}}{t_{i}^{2}}, \frac{s_{i}}{t_{i}^{3}}\Big)$.
\end{cor}

\section{Unramified abelian extension}

\noindent In this section, we construct an unramified quadratic extension of a bi-quadratic field generated from the elliptic curve $E: y^{2} = (x - a)(x + b)(x- c)$ such that each of $a,b$, and $c$ are positive and $a b + b c = a c$. We start with the following lemma, which introduces an element in the bi-quadratic field of our interest.
\begin{lem}\label{lem1}
    Let $\big(\frac{r}{t^{2}}, \frac{s}{t^{3}}\big)$ be an arbitrary point on $E: y^{2} = (x - a)(x + b)(x- c)$ defined above. Let $K$ denotes the bi-quadratic field $\mathbb{Q}(\sqrt{a b c}, \sqrt{r+(b - a - c)t^{2}})$. If $\alpha = s + t^{3} \sqrt{a b c}  \in K$, then $\alpha \mathcal{O}_{K} = \mathfrak{a}^{2}$ for some ideal $ \mathfrak{a} \subset \mathcal{O}_{K}$ whenever $\gcd(s, a b c) = 1$ and $t$ is even.  
\end{lem}
\begin{proof}
\noindent  We first notice that $P = \big(\frac{r}{t^{2}}, \frac{s}{t^{3}}\big) \in E(\mathbb{Q})$ implies $s^{2} = r^{3} + (b - a - c)r^{2}t^{2} + a b c t^{6}$. Assuming $\bar{\alpha} = s - t^{3} \sqrt{a b c}$, this in return shows that 
$$N_{K/ \mathbb{Q}}(\alpha) = \alpha^{2} {\bar{\alpha}}^{2} = (s^{2} - a b c t^{6})^{2} = (r\sqrt{r+(b - a - c)} t^{2})^{4}.$$
Without loss of generality, one can choose $\gcd(r,t) = \gcd(s,t) = 1$. Hence, if $t$ is even, both $r$ and $s$ must be odd. Now for any prime ideal $\wp \subset \mathcal{O}_{K}$, if $\alpha, \bar{\alpha} \in \wp$, then $\alpha + \bar{\alpha} = 2s, \alpha \cdot \bar{\alpha} = r^{3}+(b - a - c)r^{2}t^{2} \in \wp$. This implies that $2 \not \in \wp$ as $r^{3}+(b - a - c) r^{2}t^{2}$ is odd. Now, $s, r^{3}+(b - a - c)r^{2}t^{2} \in \wp$ implies that $a b c t^{6} \in \wp$ which is again a contradiction as $\gcd(s, a b c) = \gcd(s,t) = 1$.\\
As $\gcd(\alpha, \bar{\alpha}) = 1$, we can now conclude that $\alpha^{2} \mathcal{O}_{K} = \mathfrak{a}^{4}$ now for some ideal $\mathfrak{a} \subset \mathcal{O}_{K}$. This concludes the proof. 
\end{proof}
\noindent We conclude this section with the following result concerning the ramification of primes in a quadratic extension of $K$. 

\begin{lem}\label{lem3}
    Let $\big(\frac{r}{t^{2}}, \frac{s}{t^{3}}\big)$ be an arbitrary point on $E: y^{2} = (x - a)(x + b)(x- c)$ defined above. Let $K$ denotes the bi-quadratic field $\mathbb{Q}(\sqrt{a b c}, \sqrt{r+(b - a - c)t^{2}})$. If $\alpha = s + t^{3} \sqrt{a b c} \in K$, then $K(\sqrt{\alpha})/K$ is an unramified abelian extension if $t$ is even, $\gcd(s, a b c) = 1$, and $K$ is either imaginary or $3$ ramifies in $K/ \mathbb{Q}$.
\end{lem}
\begin{proof}
    The extension $K(\sqrt{\alpha})/K$ is quadratic, hence abelian. For the ramification of primes, we first notice that $\alpha \mathcal{O}_{K} = \mathfrak{a}^{2}$ for an ideal $\mathfrak{a}$ implies that there are no finite primes except possible primes above $2$ that can ramify in $K(\sqrt{\alpha})/K$.\\ For the ramification of primes above $2$, we first notice that $t$ even implies that $s$ is odd, hence $\alpha \equiv s \equiv 1,3 \pmod 4$. If $s \equiv 1 \pmod 4$, no primes above $2$ ramify in $K(\sqrt{\alpha})/K$. If $\alpha \equiv s \equiv 3 \pmod 4$ and $K$ is imaginary, replacing $\alpha$ by $-\alpha$ again make the extension $K(\sqrt{\alpha})/K$ unramified for all primes including infinite primes. In case, $K$ is real and $\alpha \equiv s \equiv 3 \pmod 4$, because $3$ ramifies in $K/ \mathbb{Q}$, replacing $\alpha$ by $3 \alpha$ make the extension $K(\sqrt{\alpha})/K$ unramified for all primes. This concludes the proof.  
\end{proof}
\section{The $2$-descent method}
\noindent We first recall the $2$-descent method to find the $2$-Selmer group for the elliptic curve 
\begin{equation}\label{eq1}
E_{p}: \;\;y^{2} = (x+6p)(x - 9p) (x - 18p),
\end{equation}
where $p$ is any prime. The discriminant of the elliptic curve $E_{p}$ can be observed to be $16 \cdot 2^{6} \cdot 3^{8} \cdot 5^{2} \cdot p^{6}$. Reducing $E_{p}$ modulo $11$ and modulo $13$ one can immediately observe that $|(\Tilde{E}_{p})(\mathbb{F}_{11})| = 12$ and $|(\Tilde{E}_{p})(\mathbb{F}_{13})| = 8 \text{ or } 20$. Since $E[2] \subset E_{tors}(\mathbb{Q})$ and $E_{tors}(\mathbb{Q})$ injects into both $\Tilde{E}(\mathbb{F}_{11})$ and $\Tilde{E}(\mathbb{F}_{13})$ for any elliptic curve $E$, one can see that $(E_{p})_{tors}(\mathbb{Q}) \cong \mathbb{Z}/ 2 \mathbb{Z} \times \mathbb{Z}/ 2 \mathbb{Z}$. Let $S$ be the set consisting of all finite places at which $E_{p}$ has a bad reduction, the infinite places, and the prime $2$, i.e., $S=\{p, \,2, \,3, \,5, \, \infty\}$. We define
\begin{align}
\label{pairs}
\mathbb{Q}(S,2)&=\left\{b\in\mathbb{Q}^*/(\mathbb{Q}^*)^2 :
\text{ord}_l(b)\equiv 0 ~(\bmod \text{ } {2}) ~ \text{for all primes} ~ l\not \in S\right\}\\
&=\left\{\pm 1,\; \pm 2, \;\pm 3, \;\pm 5, \;\pm p, \;\pm 6, \;\pm 10, \;\pm 2p, \;\pm 15, \;\pm 3p, \;\pm 5p, \;\pm 30, \;\pm 6p, \;\pm 10p, \;\pm 15p, \;\pm 30p \right\}.\nonumber
\end{align}
By the method of 2-descent (see \cite{Silverman}, Proposition X.$1.4$), there exists an injective homomorphism
$$ \phi: E_{p}(\mathbb{Q})/2E_{p}({\mathbb{Q}}) \longrightarrow \mathbb{Q}(S,2) \times \mathbb{Q}(S,2)$$
 defined by 
\begin{align*}
\phi(x,y) = \begin{cases}
(x + 6p, x - 9p)  & \text{if} \text{  } x\neq -6p, 9p, \\
(10, -15p) & \text{if} \text{  } x = -6p,\\
(15p,-15) & \text{if}\text{  } x = 9p, \\
(1,1) &\text{if} \text{  } x = \infty, \text{ i.e., if } (x,y) = O,
\end{cases}
\end{align*}
where $O$ is the fixed base point (also identified as the {\it point of infinity} $[0,1,0]$ in the projective plane) that acts as the identity element in the group $E_{p}(\mathbb{Q})$. 
Moreover, if $(b_1, b_2) \in \mathbb{Q}(S,2) \times \mathbb{Q}(S,2)$ is a pair that is not in the image of one of the three points $O, (-6p,0), (9p,0)$, then $(b_1,b_2)$ is the image of a point $ P = (x,y) \in E_{p}(\mathbb{Q})/2E_{p}(\mathbb{Q})$ if and only if the equations
\begin{align}
& b_1z_1^2 - b_2z_2^2 = 15p, \label{eq2}\\
& b_1z_1^2 - b_1b_2z_3^2 = 24p, \label{eq3}
\end{align}
have a solution $(z_1,z_2,z_3) \in \mathbb{Q}^* \times \mathbb{Q}^* \times \mathbb{Q}$. These equations represent a finite set of smooth curves, each isomorphic to $E_{p}$ over $\bar{\mathbb{Q}}$ together with a simply transitive algebraic group action of $E_{p}$ on each of them over $\mathbb{Q}$. These smooth curves are called \textit{homogeneous space} of $E_{p}$ defined over $\mathbb{Q}$. Computing $E_{p}(\mathbb{Q})/2E_{p}(\mathbb{Q})$ boils down to determining the existence of $\mathbb{Q}$-rational points in these spaces. Furthermore, the general failure of the Hasse-Minkowski principle motivates the definition of the Selmer group, which studies the \textit{adelic} points on the homogeneous spaces (see Remark X$.1.2$, X$.4.9$, and Section $3$ of chapter X in \cite{Silverman}).\\ 
\noindent The image of $E_{p}(\mathbb{Q})/2E_{p}({\mathbb{Q}})$ under the $2$-descent map is contained in a subgroup of $\mathbb{Q}(S,2)\times \mathbb{Q}(S,2)$ known as the \textit{$2$-Selmer group} $\mbox{Sel}_2(E_{p}/\mathbb{Q})$, which fits into an exact sequence (see Chapter X, \cite{Silverman}, Theorem X$.4.2$)
\begin{equation}\label{eq4}
0 \longrightarrow E_{p} (\mathbb{Q})/2 E_{p}(\mathbb{Q}) \longrightarrow \mbox{Sel}_2 (E_{p}/\mathbb{Q}) \longrightarrow \Sha(E_{p}/\mathbb{Q})[2]\longrightarrow 0.
\end{equation}
The elements in $\mbox{Sel}_2(E/\mathbb{Q})$ for any elliptic curve $E$ correspond to the pairs $(b_{1},b_{2})\in \mathbb{Q}(S,2) \times \mathbb{Q}(S,2)$ such that the system of equations (\ref{eq2}) and (\ref{eq3}) has non-trivial local solutions in $\mathbb{Q}_{l}$ at all primes $l$ of $\mathbb{Q}$ including infinity. Note that $\# \,E_{p}(\mathbb{Q})/2 E_{p}(\mathbb{Q})=2^{2+r(E_{p})}$. It is customary to write $\#  \,\mbox{Sel}_2 (E_{p}/\mathbb{Q}) = 2^{2+s(E_{p})}$, and refer to $s(E_{p})$ as the $2$-Selmer rank. We have
\begin{equation}\label{eq5}
0 \leq r(E_{p}) \leq  s(E_{p}).
\end{equation}
From the inequality above as well as the exact sequence (\ref{eq4}), one can say that the $2$-Selmer group provides information about both $E_{p}(\mathbb{Q})$ and $\Sha(E_{p}/\mathbb{Q})[2]$.  
\section{Local solutions of the homogeneous space}
\noindent In this section, we examine the properties of the $l$-adic solutions for (\ref{eq2}) and (\ref{eq3}) that are associated with the $2$-Selmer group. We use these properties later to bound the size of the $2$-Selmer group. Any $l$-adic number $a$ can be written as $a = l^{n} \cdot u \text{ where } n \in \mathbb{Z}$, $u \in \mathbb{Z}_{l}^{*}$. Notice that the $l$-adic valuation $v_{l}(a)$ of $a$ is just $n$. We first prove the following result for all primes $l \neq 5$.  
\begin{lem}\label{lem4}
Suppose (\ref{eq2}) and (\ref{eq3}) have a solution $(z_1,z_2,z_3) \in \mathbb{Q}_{l} \times \mathbb{Q}_{l} \times \mathbb{Q}_{l}$ for any odd prime $l \neq 5$. If $v_{l}(z_{i}) < 0$ for any one $i \in \{1,2,3\}$, then $v_{l}(z_{1})=v_l(z_2)=v_l(z_3)=-k < 0$ for some integer $k$. 
\end{lem}
\begin{proof}
Let $z_{i} = l^{k_{i}}u_{i}$, where $k_{i} \in \mathbb{Z}$ and $u_{i} \in \mathbb{Z}_{l}^{*}$ for $i \in \{1,2,3\}$. Then $v_{l}(z_{i}) = k_{i}$ for all $i \in\{1,2,3\}$. \\
\noindent Suppose $k_{1} < 0$. Then from (\ref{eq2}) one can get that $$b_{1}u_{1}^{2} - b_{2}u_{2}^{2}l^{2(k_{2}-k_{1})} = 15pl^{-2k_{1}}.$$ If $k_{2}>k_1$, then $l^{2}$ must divide $b_{1}$, a contradiction as $b_{1}$ is square-free. Hence $k_{2} \leq k_1<0$. Now if  $k_{2} < k_{1} < 0$ then again from (\ref{eq2}) we get $$b_{1}u_{1}^{2}l^{2(k_{1}-k_{2})} - b_{2}u_{2}^{2} = 15pl^{-2k_{2}},$$ which implies $l^{2}$ must divide $b_{2}$, a contradiction again. Hence if $k_{1} < 0$, then we have $k_{1} = k_{2} = -k < 0$ for some integer $k$. For $k_2<0$, one similarly gets $k_1=k_2=-k<0$.\\
From (\ref{eq3}), we have $$b_{1}u_{1}^{2} - b_{1}b_{2}u_{3}^{2}l^{2(k_{3} - k_{1})} = 24pl^{-2k_{1}}.$$ If $k_1<0$ and $k_3>k_1$, then $l^{2}$ must divide $b_{1}$, a contradiction as before. Hence $k_3\leq k_1<0$ if $k_1<0$. For $k_3 < k_{1} <0$, 
we can rewrite the above equation as 
\begin{equation}
\label{new1}
b_{1}u_{1}^{2}l^{2(k_{1} - k_{3})} - b_{1}b_{2}u_{3}^{2} = 24pl^{-2k_{3}},
\end{equation} which implies $l^{2}$ must divide $b_{1}b_{2}$ i.e. $l$ divides both $b_{1}$ and $b_{2}$. If $l = 2,3$ or $p$, then from (\ref{new1}) we arrive at the contradiction that $l^{3}$ divides $b_{1}b_{2}$ whereas $b_1$ and $b_2$ are square-free. Hence $k_{1} < 0$ implies $k_{3} = k_{1} < 0$. Together, we obtain $k_1=k_2=k_3=-k<0$ for some integer $k$ if $k_1<0$.\\
Now if $k_{2} < 0$ but $k_{1} \geq 0$, then from \ref{eq2}, we get $$b_{1}u_{1}^{2}l^{2(k_{1}-k_{2})} - b_{2}u_{2}^{2} = 15pl^{-2k_{2}} \implies b_{2} \equiv 0 \pmod {l^{2}},$$ a contradiction. Hence $k_{2} < 0 \implies k_{1} < 0$ and we are back to the first case which implies $k_{1} = k_{2} = k_{3} = -k$ for some integer $k \geq 0$.\\
Similarly if $k_{3} < 0$ but $k_{1} \geq 0$, we get from \ref{eq3}, for $l\neq 5$, $$b_{1}u_{1}^{2}l^{2(k_{1} - k_{3})} - b_{1}b_{2}u_{3}^{2} = 24pl^{-2k_{3}} \implies l^{3} \text{ divides } b_{1}b_{2}, \text{ a contradiction }.$$ So again we observe that $k_{3} < 0$ implies $k_{1} < 0$ and hence $k_{1} = k_{2} = k_{3} = -k$ for some positive integer $k$ again. This concludes the proof.  
\end{proof}
\noindent An immediate observation from the proof above is that the conclusion of Lemma \ref{lem4} stays true even for the case $l=5$ as long as $b_{1}b_{2} \not \equiv 0 \pmod {25}$. We conclude this section with the following lemma for the case $l=5$.
\begin{lem}\label{lem4.1}
    Suppose $5$ divides $\gcd(b_{1},b_{2})$ and the corresponding equations (\ref{eq2}) and (\ref{eq3}) have a solution $(z_1,z_2,z_3) \in \mathbb{Q}_{5} \times \mathbb{Q}_{5} \times \mathbb{Q}_{5}$. If $v_{5}(z_{i}) < 0$ for any one $i \in \{1,2\}$, then $v_{5}(z_{1})=v_5(z_2)=v_5(z_3)=-k < 0$ for some integer $k$. If $v_{5}(z_{3}) < 0$ but $v_{5}(z_{1}) \neq v_{5}(z_{3})$, then $v_{5}(z_{3}) = -1$ and $v_{5}(z_{1}) \geq 0$.
\end{lem}
\begin{proof}
    The first part of the statement follows similarly to the proof of Lemma \ref{lem4}. Hence we only focus on the case $k_{3} = v_{5}(z_{3}) < 0$. Now if $k_{1} = v_{5}(z_{1}) \leq k_{3} < 0$ then that will imply from (\ref{eq3}), $$b_{1}u_{1}^{2} - b_{1}b_{2}u_{3}^{2}5^{2(k_{3} - k_{1})} = 24 \cdot p \cdot 5^{-2k_{1}} \implies b_{1} \equiv 0 \pmod {25},$$ a contradiction. Now if $k_{3} < k_{1}$, and $k_{3} \leq -2$, we again get from (\ref{eq3}), $$b_{1}u_{1}^{2}5^{2(k_{1}-k_{3})} - b_{1}b_{2}u_{3}^{2} = 24 \cdot p \cdot 5^{-2k_{3}} \implies b_{1}b_{2} \equiv 0 \pmod {5^{3}},$$ a contradiction. Hence $k_{3} < 0$ implies $k_{3} = -1$ and $k_{1} > k_{3}$ i.e. $k_{1} \geq 0$. Now the result follows.
\end{proof}

\section{Bounding the size of the $2$-Selmer group of $E_{p}$}

\noindent In this section, we bound the size of the $2$-Selmer group of the elliptic curve $E_{p}: y^{2} = (x+6p)(x-9p)(x-18p)$. The $2$-Selmer group $\mbox{Sel}_2 (E_{p}/\mathbb{Q})$ consists of those pairs $(b_1, b_2)$ in $\mathbb{Q}(S,2) \times \mathbb{Q}(S,2)$ for which equations (\ref{eq2}) and (\ref{eq3}) have an $l$-adic solution at every place $l$ of $\mathbb{Q}$. 
We now limit the size of  $\mbox{Sel}_2 (E_{p}/\mathbb{Q})$ by ruling out local solutions for specific pairs $(b_{1},b_{2})$ by exploiting the results of the previous sections. 
\begin{lem}\label{lem5}
Let $(b_{1},b_{2}) \in \mathbb{Q}(S,2) \times \mathbb{Q}(S,2)$. Then\\
$(i)$ The corresponding homogeneous space will have no $l$-adic solution for the case $l = \infty$ if $b_{1} < 0$.\\
$(ii)$ If $\gcd(b_{1},b_{2}) \equiv 0 \pmod 2$, the corresponding homogeneous space will not have $2$-adic solutions. \\ 
\end{lem}
\begin{proof} $(i)$ The result follows from the observation that $b_{1} < 0$ in equations (\ref{eq2}) and (\ref{eq3}) either implies $15p < 0$ when $b_{2} > 0$ or $24p < 0$ when $b_{2} < 0$.\\
\noindent $(ii)$ Suppose $\gcd(b_{1},b_{2})$ is even. If moreover, $v_{2}(z_{i}) < 0$ for some $i \in \{1,2,3\}$, then from Lemma \ref{lem4}, we can say that $v_{2}(z_{i}) = -k$ for some positive integer $k$ and for all $i \in \{1,2,3\}$. From (\ref{eq3}) we can now get $$b_{1}u_{1}^{2} - b_{1}b_{2}u_{3}^{2} = 24 \cdot p \cdot 2^{2k},$$ where $z_{i} = u_{i} \cdot 2^{-k}, u_{i} \in {\mathbb{Z}_{2}}^{*}$ for all $i \in \{1,2,3\}$. This immediately implies that $b_{1} \equiv 0 \pmod 4$, a contradiction. Hence $v_{2}(z_{i}) \geq 0$ for all $i \in \{1,2,3\}$ which is again a contradiction as then (\ref{eq2}) implies $15p \equiv 0 \pmod 2$. Hence the result follows. 
\end{proof}
\noindent Let $A = \{(1,1), (10, -15p), (15p, -15), (6p,p)\} \subset \mathbb{Q}(S,2) \times \mathbb{Q}(S,2)$ be the image of $(E_{p})_{tors}(\mathbb{Q})$ under the map $\phi$. We notice the fact that it is enough to check the local solutions for homogeneous space corresponding to only one $(b_{1},b_{2})$ for each coset in $(\mathbb{Q}(S,2) \times \mathbb{Q}(S,2)) / A$ as the existence or non-existence of local solutions for one pair $(b_{1},b_{2})$ determines the same for all four members of the coset. This observation leads us to the following lemma.
\begin{lem}\label{lem6}
    $(i)$ The homogeneous space corresponding to $(b_{1}, b_{2})$ can not have $l$-adic solution for all primes $l \leq \infty$ if $b_{2}$ is even.\\
    $(ii)$ Let $\gcd(b_{1},b_{2}) \equiv 0 \pmod 3$. Then $(b_{1}, b_{2}) \equiv (c_{1},c_{2}) \pmod A$ such that $\gcd(c_{1},c_{2})$ is not divisible by either $3$ or $p$.\\
    $(iii)$ Let $\gcd(b_{1},b_{2}) \equiv 0 \pmod p$. Then $(b_{1}, b_{2}) \equiv (c_{1},c_{2}) \pmod A$ such that $\gcd(c_{1},c_{2})$ is not divisible by either $3$ or $p$.
\end{lem}
\noindent Before proving Lemma \ref{lem6}, we first explain the importance of it in the $2$-descent method here. Once proved, Lemma \ref{lem6} will imply that without loss of generality, we can assume $b_{2}$ is odd and $\gcd(b_{1},b_{2})$ is not divisible by either $3$ or $p$ while looking for local solutions as long as $(b_{1},b_{2})$ satisfies the conditions above. 
\begin{proof}
    $(i)$ Let us suppose $b_{2}$ is even. Then from Lemma \ref{lem5}, we can deduce that $b_{1}$ is odd. But then $(b_{1}, b_{2})$ and $(6pb_{1}, pb_{2})$ belong to the same coset of $A$ and the homogeneous space corresponding to $(6pb_{1}, pb_{2})$ has no $2$-adic solution from Lemma \ref{lem5}. Hence the result follows.\\
    $(ii)$ We prove this by observing that modulo $A$, the coset of $(b_{1},b_{2})$ consists of\\ $\{(b_{1},b_{2}), (10b_{1}, -15pb_{2}), (15pb_{1}, -15b_{2}), (6pb_{1}, pb_{2})\}$. The result is now evident from the fact that one of $(15pb_{1}, -15b_{2}), (6pb_{1}, pb_{2})$ will always satisfy the condition of the statement depending on whether $3$ or $p$ divides $b_{1},b_{2}$ or not.\\
    $(iii)$ The proof is identical to the proof of the previous result.  
\end{proof} 
\noindent We are now left with $256$ points. As mentioned above, from now on we only consider $(b_{1}, b_{2})$ such that $b_{2}$ is odd and $\gcd(b_{1}, b_{2})$ is not divisible by either $3$ or $p$. This leaves us with $64$ possible pairs to look into. The following lemma helps us omit some of those $64$ possibilities.
\begin{lem}\label{lem7}
    Let us assume $p$ is a prime such that $p \equiv 17, 113 \pmod {120}$. Then\\
    $(i)$ The homogeneous space corresponding to $(b_{1},b_{2})$ has no $3$-adic solution if $b_{1} = 6, 15, 3p$ or $30p$.\\
    $(ii)$ The homogeneous space corresponding to $(b_{1},b_{2})$ has no $3$-adic solution if $b_{2} = 3, 15p, -15$ or $-3p$.\\
    $(iii)$  The homogeneous space corresponding to $(b_{1},b_{2})$ has no $p$-adic solution if $b_{1} = 3p, 5p, 6p$ or $10p$.\\
    $(iv)$ The homogeneous space corresponding to $(b_{1},b_{2})$ has no $p$-adic solution if $b_{2} = \pm 3p, \pm 15p$.\\
    $(v)$ Let $5$ divides $\gcd(b_{1},b_{2})$. Then the homogeneous space corresponding to $(b_{1},b_{2})$ has no $5$-adic solution if $\Big(\frac{b_{1}b_{2}/25}{5}\Big) = 1$. 
\end{lem}
\begin{proof}
    We start the proof by first noticing the fact $p \equiv 17, 113 \pmod {120}$ implies $\big(\frac{2}{p}\big) = 1$ and $\big(\frac{3}{p}\big) = \big(\frac{5}{p}\big) = -1$. The idea behind each proof is to manipulate this observation and (\ref{eq2}).\\
    $(i)$ From Lemma \ref{lem4}, we know that if $(z_{1}, z_{2})$ is a solution of (\ref{eq2}), and $v_{3}(z_{i}) < 0$ for any $i \in \{1,2
    \}$, then $v_{3}(z_{1}) = v_{3}(z_{2}) = -k$ for some positive integer $k$. This transform (\ref{eq2}) into $b_{1}u_{1}^{2} -b_{2}u_{2}^{2} = 15 \cdot p \cdot 3^{2k}$, where $u_{i}$ denotes $3$-adic units. From here, it is evident that $b_{1}$ can not be divisible by $3$ as $\gcd(b_{1},b_{2}) \not \equiv 0 \pmod 3$. Hence $b_{1} \equiv 0 \pmod 3 \implies z_{2} \equiv 0 \pmod 3$. This in turn implies $$\frac{b_{1}}{3}z_{1}^{2} \equiv 5p \pmod 3 \implies \Big(\frac{b_{1/3}}{3}\Big) = \Big(\frac{2p}{3}\Big) = 1.$$ Hence, if $3$ divides $b_{1}$, then $b_{1}$ can only be one of $3, 6p, 30$ or $15p$ as a necessary condition for the existence of $3$-adic solutions. Now the result follows.\\
    $(ii)$ In a very similar manner to the previous method, we notice from (\ref{eq2}) that if $b_{2} \equiv 0 \pmod 3$, then $$-\frac{b_{2}}{3}z_{2}^{2} \equiv 5p \pmod 3 \implies \Big(\frac{b_{2/3}}{3}\Big) = - \Big(\frac{2p}{3}\Big) = -1.$$ Hence $b_{2}$ can only assume one of the values from $15, 3p, -3$ or $-15p$ for the existence of $3$-adic solutions. The result now follows immediately. \\
    $(iii)$ Applying the same method used in the first part of this proof, we can immediately see that if $p$ divides $b_{1}$, then $p$ must divide $z_{2}$ from (\ref{eq2}) and hence $\Big(\frac{b_{1}/p}{p}\Big) = \Big(\frac{15}{p}\Big) = 1 \implies b_{1} = p,2p, 15p, 30p$. This concludes the proof. \\
    $(iv)$ Similar to the previous case, one can notice that $b_{2} \equiv 0 \pmod p$ implies $\Big(\frac{-b_{2}}{p}\Big)$ is a quadratic residue modulo $p$ i.e. $b_{2} = \pm p$ or $\pm 15p$. The result now follows immediately.\\
    $(v)$ Suppose $\gcd(b_{1},b_{2}) \equiv 0 \pmod 5$. Then from Lemma \ref{lem4} and, Lemma \ref{lem4.1}, we can rewrite (\ref{eq3}) as either $b_{1}u_{1}^{2} - b_{1}b_{2}u_{3}^{2} = 24 \cdot p \cdot 5^{2k}$ or $b_{1}z_{1}^{2} - \frac{b_{1}b_{2}}{25}u_{3}^{2} = 24p$ where $u_{i} \in \mathbb{Z}_{5}^{*}, z_{i} \in \mathbb{Z}_{5}$ for $i \in \{1,3\}$ and $k$ is a positive integer. The solution to the first equation implies $b_{1} \equiv 0 \pmod {5^{2}}$, a contradiction.  Whereas a solution to the second equation implies $\Big(\frac{b_{1}b_{2}/25}{5}\Big) = \Big(\frac{-24p}{5}\Big) = \Big(\frac{p}{5}\Big) = -1$. The result now follows immediately. 
\end{proof}
\noindent Now, we are left with the following six cosets.
\begin{table}[htbp] 
\caption{ Distribution of points}
\centering

\begin{tabular}{ |c|c |c|c| }
\hline

 \begin{tabular}{@{}c@{}}Representative\\  of  a coset of A \end{tabular}& Points in the coset Of A & \begin{tabular}{@{}c@{}}Representative\\  of  a coset of A \end{tabular} & Points in the coset Of A \\ 
 \hline

$(1,15)$ & $(10,- p),(15p,-1),(6p,15p)$ & $(2p,15)$& $ (30, -1),(5p, -p),(3, 15 p)$\\ 
 \hline
$(3,\pm p)$ & $(30,\mp 15), (5p, \mp 15p), (2p, \pm 1)$ & $ (3,\pm 5p)$ & $(2p, \pm 5), (30, \mp 3), (5p, \mp 3p)$\\
\hline

\end{tabular}

\end{table}

\noindent We rule out the possibility of five of the six cosets mentioned above in the following lemma. 

\begin{lem}\label{lem8}
    Let us suppose that $p \equiv 17, 113 \pmod {120}$. Then the homogeneous space corresponding to $(b_{1},b_{2})$ has no $5$-adic solution if $(b_{1},b_{2}) \in \{(1,15), (2p,15), (3, \pm 5p)\}$. Moreover, The homogeneous space corresponding to $(3,-p)$ has no $3$-adic solution.
\end{lem}

\begin{proof}
    The first part of the result follows from the following observation. If $5$ does not divide $b_{1}$ while dividing $b_{2}$, then from \ref{eq2} one can get that $z_{1} \equiv 0 \pmod 5$ and $z_{2} \not \equiv 0 \pmod 5$. This, in return implies that $\big(\frac{b_{2}/5}{5}\big) = \big(\frac{3p}{5}\big) = 1$, a condition that rules out the cases $b_{2} = 15, \pm 5p$ if $5$ does not divide $b_{1}$. Hence the result follows. \\
    For the case $(b_{1},b_{2}) = (3,-p)$, we first observe that subtracting \ref{eq3} from \ref{eq2} yield into the equation $-3pz_{3}^{2} + pz_{2}^{2} = -9p$. Noticing the fact that $\frac{z_{2}}{3} \in \mathbb{Z}_{3}^{*}$, one can reach to the conclusion that $\big(\frac{z_{2}}{3}\big)^{2} \equiv -1 \pmod 3$, a contradiction. Hence the result follows.  
\end{proof}

\section{A homogeneous space with everywhere local solution}
\noindent In this section, we prove that for the elliptic curve $E_{p}: y^{2} = (x+6p) (x-9p) (x-18p)$ where $p \equiv 17, 113 \pmod {120}$, the homogeneous space corresponding to $(3,p)$ has local solutions for all places $l \leq \infty$. We prove the following lemma regarding that. \\

\begin{lem}\label{lem9}
Equations (\ref{eq2}) and (\ref{eq3}) have a local solution in $\mathbb{Q}_{l}$ for every prime $l$ for $(b_1,b_2)=(3,p)$.
\end{lem}
\begin{proof}
A twist of a smooth projective curve $C/ \mathbb{Q}$ is a smooth curve $C^{\prime}/ \mathbb{Q}$ that is isomorphic to $C$ over $\bar{\mathbb{Q}}$ (see \cite{Silverman}, Section X$.2$). We use the following version of Hensel's lemma to show that for $(b_{1},b_{2}) = (3,p)$, equations (\ref{eq2}) and (\ref{eq3}) have a local solution in $\mathbb{Q}_{l}$ for every place $l$. Hensel's lemma states that if $f(x)$ is a polynomial with coefficients that are $l$-adic integers and $f(x_{1}) \equiv 0 \pmod l$ for $x_{1} \in \mathbb{Z}$, then if $f^{\prime}(x_{1}) \not \equiv 0 \pmod l$, there exists a $l$-adic integer $x$ with $x \equiv x_{1} \pmod l$ such that $f(x) = 0$ (see \cite{Wash}, Theorem A$.2$). \\
Every homogeneous space of $E_{p}$ is a twist of $E_{p}$ (see \cite{Silverman}, Proposition $3.2$). First, we consider $l > 5$, $l \neq p$. Suppose $C$ is the homogeneous space given by (\ref{eq2}) and (\ref{eq3}) corresponding to the pair $(3,p)$. Then $C$ is a twist of $E_{p}$, and in particular, it has genus $1$. By the Hasse-Weil bound, we have 
$$\# C(\mathbb{F}_{l})\;\geq \;1+l-2\sqrt{l}\;\geq \;2 \qquad \mbox{ for } l > 5, \text{ } l \neq p.$$
We can choose a solution $(z_{1},z_{2},z_{3}) \in \mathbb{F}_{l} \times \mathbb{F}_{l} \times \mathbb{F}_{l}$ such that not all three of them are zero modulo $l$. Now $z_{1} \equiv z_{2} \equiv 0 \pmod l$ implies $l^{2}$ divides $15p$, a contradiction. Similarly, $z_{1} \equiv z_{3} \equiv 0 \pmod l$ implies $24p \equiv 0 \pmod {l^{2}}$, contradiction again. One can now suitably choose two of $z_{1},z_{2}$ and $z_{3}$ to convert equations (\ref{eq2}) and (\ref{eq3}) into one single equation of one variable with a simple root over $\mathbb{F}_{l}$. That common solution 
can then be lifted to $\mathbb{Q}_{l}$ via Hensel's lemma.\\
\noindent For $l = p$, one can notice that fixing $z_{1} = 0$ turn equations (\ref{eq2}) and (\ref{eq3}) with simple roots $z_{i} \not \equiv 0 \pmod p$ for $i \in \{2,3\}$ and hence can be lifted to a solution in $\mathbb{Z}_{p}$ via Hensel's lemma.\\
\noindent For $l=3$, first, we fix $z_{2} = z_{3} = 1$ to reduce equations (\ref{eq2}) and (\ref{eq3}) to equations of only one variable. Now observing $z_{1} \not \equiv 0 \pmod 3$ is a simple root that Hensel's lemma can lift, we conclude this case.\\
\noindent For the case $l=5$, notice that if $p \equiv 2 \pmod 5$, then fixing $z_{2} = 1$ and $z_{3} = 2$ gives the equation $3z_{1}^{2} \equiv 2 \pmod 5$ for which $z_{1} \equiv 2 \pmod 5$ is a simple root and hence can be lifted by Hensel's lemma. Similarly, for the case $p \equiv 3 \pmod 5$, we can fix $z_{2} = 1$ and $z_{3} = 2$ and then notice that $z_{1} \equiv 1 \pmod 5$ is then a simple root and hence can be lifted. This concludes the proof for $l=5$.\\ 
Finally for the case $l=2$, choose $z_{2} = 2$ and $z_{3} = 1$. This turns (\ref{eq2}) and (\ref{eq3}) into the following;
\begin{align}
    3z_{1}^{2} - 4 = 15p \implies 3z_{1}^{2} \equiv 3 \pmod 8,\\
    3z_{1}^{2} - 3pz_{3}^{2} = 24p \equiv 0 \pmod 8 \implies 3z_{1}^{2} \equiv 3z_{3}^{2} \pmod 8.
\end{align} In both the cases we now have $z_{1}^{2} \equiv 1 \pmod 8$ which by Hensel's lemma can be lifted to $\mathbb{Q}_{2}$. Hence proved.   
\end{proof}
\section{Class number divisibility and $2$-Selmer rank}
\noindent We are now in a position to prove Theorem \ref{mainthm}. Identifying $a,b$ and $c$ in Lemma \ref{lem1} with $9p, 6p$ and $18p$ respectively, we first notice that the elliptic curve $E_{p}: y^{2} = (x+6p) (x-9p) (x-18p)$ mentioned in Theorem \ref{mainthm} is same as the elliptic curve $E$ in Lemma \ref{lem1}. 
\begin{proof}[Proof of Theorem \ref{mainthm}]
    Let $E_{p}: y^{2} = (x+6p) (x-9p) (x-18p)$ be the elliptic curve where $p \equiv 17, 113 \pmod {120}$. the first part of the statement is now evident as Lemma \ref{lem9} proves that the homogeneous space corresponding to $(3,p)$ is the only one with local solutions for all primes $l \leq \infty$. This proves that the $2$-Selmer rank of $E_{p}$ is exactly one. \\ 
    \noindent Now for the second part, we get that $\mathbb{Q}(\sqrt{abc}, \sqrt{r+ (b-a-c)t^{2}}) = \mathbb{Q}(\sqrt{3p}, \sqrt{r-21pt^{2}})$. We notice that $\gcd(s,3p) = 1 \implies \gcd(s,abc) = 1$. Also $3$ ramifies in $\mathbb{Q}(\sqrt{3p}, \sqrt{r-21pt^{2}})$. Hence, from Lemma \ref{lem3}, the result follows.  
\end{proof}
\begin{proof}[Proof of Corollary \ref{cor1}]
    We first notice that $p \equiv 53, 77 \pmod {120} \implies p \equiv 5 \pmod 8$, which implies $\big(\frac{2}{p}\big) = -1$. From the $2$-descent method in Section $5$, one can see Lemma \ref{lem7}.$(iii)$ is the only argument where we have used the fact $\big(\frac{2}{p}\big) = 1$. So, in this case, the homogeneous spaces corresponding to each possible point except $(5p, \pm 5), (1,15), (1,5p)$ and $(3,p)$ in $\mathbb{Q}(S,2) \times \mathbb{Q}(S,2)$ are still ruled out from having $l$-adic solutions for all primes $l \leq \infty$.\\
    \noindent Firstly, for $(5p, \pm 5)$, we notice that $v_{p}(z_{1}) < 0$ is not possible, and hence $v_{p}(z_{i}) \geq 0$ for $i \in \{1,2\}$ from Lemma \ref{lem4}. This, in turn implies that $$z_{2} \equiv 0 \pmod p \implies z_{1}^{2} \equiv 3 \pmod p, \text{ a contradiction}.$$
    In a very similar way, one can notice that local solutions for the homogeneous space corresponding to $(1, 15)$ implies $z_{1}^{2} \equiv -p \pmod 5$ and to $(1, 15)$ implies $z_{1}^{2} \equiv -3 \pmod 5$, both contradictions.\\ 
    \noindent For the homogeneous space corresponding to $(3,p)$, we first observe that for $l = p$, $v_{p}(z_{1}) < 0$ is not a possibility from Lemma \ref{lem4}. Hence, $v_{p}(z_{1}), v_{p}(z_{3}) \geq 0$. But then from \ref{eq2}, we get that $3z_{1}^{2} - 3p z_{3}^{2} = 24p$ which implies $z_{1} \equiv 0 \pmod p$ and hence $z_{3}^{2} \equiv -8 \pmod p$, a contradiction as $\big(\frac{2}{p}\big) = -1$ now that $p \equiv 5 \pmod 8$. This proves that $(3,p)$ will not have any $p$-adic solution and hence shows that the $2$-Selmer rank of $E_{p}$ is zero if $p \equiv 53, 77 \pmod {120}$. This ends the proof of the result.    
\end{proof}
\noindent We can now prove Corollary \ref{cor2} by explicitly constructing an infinite family of bi-quadratic fields with even class numbers from the points of $E_{p}$. The Parity conjecture ensures $r(E_{p}) = 1$ for $p \equiv 17, 113 \pmod {120}$ as the $2$-Selmer rank of $E_{p}$ is $1$. We, though, use a more straightforward characterization for the existence of points of infinite order below.
\begin{proof}[Proof of Corollary \ref{cor2}]
    We first notice that the $t_{0}$ even implies the Mordell-Weil rank of $E_{p}$ is at least $1$. If $P_{0} = \big(\frac{r_{0}}{t_{0}^{2}}, \frac{s_{0}}{t_{0}^{3}}\big)$ is a point of infinite order in $E_{p}(\mathbb{Q})$ with $t_{0}$ even, one can immediately observe from the duplication formula that $$x(2P) = \Big(\frac{r_{1}}{t_{1}^{2}}\Big) = \Big(\frac{r_{0}^{4} - 8 \cdot 972p^{3} \cdot r_{0}t_{0}^{6} + 4 \cdot 21p \cdot 972p^{3} \cdot t_{0}^{8}}{4s_{0}^{2}t_{0}^{2}}\Big).$$ This implies $t_{1}$ is even as $r_{0}^{4} - 8 \cdot 972p^{3} \cdot r_{0}t_{0}^{6} + 4 \cdot 21p \cdot 972p^{3} \cdot t_{0}^{8}$ is odd when $t_{0}$ is even. A similar calculation involving duplication formula shows that $\gcd(s_{0}, 3p) = 1 \implies \gcd(s_{1},3p) = 1$ too. This is because if $s$ is divisible by either $3$ or $p$, then so is $r$ and hence $$\gcd(s_{0}, 3p) = 1 \implies \gcd(r_{0}, 3p) = 1 \implies \gcd(r_{1}, 3p) = 1 \implies \gcd(s_{1}, 3p) = 1.$$  Also, $3$ ramifies in $K_{1} / \mathbb{Q}$. Hence the construction of an unramified abelian extension for the bi-quadratic field $K_{1} = \mathbb{Q}(\sqrt{3p}, \sqrt{r_{1} - 21pt_{1}^{2}})$ follows through Lemma \ref{lem3}. The proof now follows through an inductive argument. 
\end{proof}
\noindent We conclude this section with a small table containing elliptic curves $E_{p}$ for primes $p \equiv 17,113 \pmod {120}$ with their respective $2$-Selmer ranks and bi-quadratic fields with corresponding class numbers in support of our result. Each example below can generate an infinite family of bi-quadratic fields with even class numbers by Corollary \ref{cor2}. Computations have been done in Magma \cite{Magma} software. While we define $\#  \,\mbox{Sel}_2 (E_{m,p}/\mathbb{Q}) = 2^{2+s(E_{m,p})}$, and refer to $s(E_{m,p})$ as the $2$-Selmer rank in the following table, $2+s(E_{m,p})$ will denote the $2$-Selmer rank in Magma.\\

\begin{table}[htbp] 
\caption{$2$-Selmer rank and Class numbers}
\centering

\begin{tabular}{ |c|c |c|c| }
\hline

 \begin{tabular}{@{}c@{}}$p$ \end{tabular}& $s(E_{p})$ & \begin{tabular}{@{}c@{}}$P = \big(\frac{r}{t^{2}}, \frac{s}{t^{3}}\big) \in E_{p}(\mathbb{Q})$ \end{tabular} & $h(K)$ \\ 
 \hline

$17$ & $1$& $\left( \dfrac{5257}{16},\; \dfrac{83581}{64}\right)$ & $2560$\\
 \hline
$113$ & $1$ & $\left( \dfrac{103952278736903209}{48681040752400},\; \dfrac{6066145443591398465996627}{339656383916830232000}\right)$ & No information\\
\hline

\end{tabular}

\end{table}
\noindent We conclude by noting that the choice of $-6p, 9p$ and $18p$ is solely to ensure the ramification of the prime $3$ in the extension $K / \mathbb{Q}$ as well as to ensure $ab+bc = ac$ as mentioned in Lemma \ref{lem1}. A different choice satisfying the conditions of Section $2$ would still result in an unramified abelian extension and hence even class number. 
\section*{Acknowledgement}
\noindent The first author would like to acknowledge DST-SERB for providing the grant through the Start-Up Research Grant (SRG/2020/001937) as well as BITS-Pilani, Hyderabad Campus, for providing amenities. The second author would like to acknowledge the fellowship (File No:09/1026(0029)/2019-EMR-I) and amenities provided by the Council of Scientific and Industrial Research, India (CSIR) and BITS-Pilani, Hyderabad. The third author would like to acknowledge the fellowship (File No:09/1026(0036)/2020-EMR-I) and amenities provided by the Council of Scientific and Industrial Research, India (CSIR) and BITS-Pilani, Hyderabad.

\end{document}